\documentclass[reqno]{amsart}
\usepackage{amsmath,amsthm,amssymb,hyperref,graphicx,mathrsfs,mathtools} 

\newtheorem{theorem}{Theorem}
\newtheorem{corollary}{Corollary}
\newtheorem{lemma}{Lemma}
\newtheorem{thmx}{Theorem}
\newtheorem{remark}{Remark}

\usepackage[left=1in, right=1in]{geometry}

\allowdisplaybreaks

\begin{document}

\title[A REFINEMENT OF CAUCHY'S THEOREM ON THE ZEROS OF QUATERNION POLYNOMIALS]{A REFINEMENT OF CAUCHY'S THEOREM ON THE ZEROS OF QUATERNION POLYNOMIALS}
\author{N. A. Rather$^{1}$}
\author{Danish Rashid$^{2}$}
\author{Tanveer Bhat$^3$}
\address{$^{1,2,3}$ Department of Mathematics, University of Kashmir, Srinagar-190006, India}
\email{$^1$dr.narather@gmail.com $^2$danishmath1904@gmail.com $^3$tanveerbhat054@gmail.com}

\begin{abstract}
In this paper, we shall present an interesting and significant refinement of a classical result of  Cauchy about the moduli of the zeros of a quaternionic polynomial. As an application of this result we shall obtain zero-free regions of polynomials having quaternionic coefficients.
\end{abstract}

\maketitle

\section{\textbf{INTRODUCTION }}
Concerning the location of zeros of the quaternion polynomials, the first study was done by Eilenberg and Niven \cite{ES}, \cite{NI}. After these fundamental works, the question of locating the zeros of quaternion polynomials have been vastly investigated. Let us first introduce the background information about the quaternions and quaternionic polynomials. The quaternions were first introduced by an Irish mathematician Sir William Rowan Hamilton in 1843. The quaternions are mathematical entities used to represent rotations in $3$-D space. They extend the concept of complex numbers by adding two more imaginary units providing a concise way to perform spatial rotations. The quaternion number system is represented by the letter $\mathbb{H}$ and are generally represented as $q = \alpha + i\beta + j\gamma + k\delta \in \mathbb{H}$, where $\alpha, \beta, \gamma, \delta \in R$ and $i, j, k$ are the fundamental quaternion units such that $i^{2} = j^{2} = k^{2} = ijk = -1$. 
Depending upon the position of the coefficients, the quaternion polynomial of degree n  in indeterminate $q$ is defined as 
$f(q) = q^{n} +q^{n-1}a_{1}+ \cdots +qa_{n-1} + a_n$ or $g(q) = q^{n} +a_{1}q^{n-1}+ \cdots +a_{n-1}q + a_n.$\\
\textbf{The quaternion Companion Matrix}: The $n \times n$ companion matrix of a monic quaternion polynomial of the form $f(q) = q^{n} +q^{n-1}a_{1}+...+ qa_{n-1} + a_n,$ is given by
$$C_f=\begin{bmatrix}
0& 0& 0& \cdots & 0& -a_n\\
1& 0& 0& \cdots & 0& -a_{n-1}\\
0& 1& 0& \cdots & 0& -a_{n-2}\\
0& 0& 1& \cdots & 0& -a_{n-3}\\
\vdots\\
0& 0& 0& \cdots& 1& -a_{1}\\
\end{bmatrix},$$
whereas, the $n \times n$ companion matrix for a monic quaternion polynomial of the form $g(q) = q^{n} + a_{1}q^{n-1} + \cdots +a_{n-1}q + a_n,$ is given by
$$C_g=\begin{bmatrix}
0& 1& 0& \cdots & 0\\
0& 0& 1& \cdots & 0\\
\vdots\\
0& 0& 0& \cdots & 1\\
-a_{n}& -a_{n-1}& -a_{n-2}& \cdots& -a_{1}\\
\end{bmatrix}.$$

\textbf{Right Eigenvalue:} Given an $n \times n$ matrix $A = [a_{\mu\nu}]$ of quaternions, $\lambda \in \mathbb{H}$ is called the right eigenvalue of A, if $Ax = x\lambda$ for some non zero eigenvector $x = [x_1, x_2, \cdots x_n]^T$ of quaternions.\\
 
\textbf{Left Eigenvalue:}  Given an $n \times n$ matrix $A = [a_{\mu\nu}]$ of quaternions, $\lambda \in \mathbb{H}$ is called the left eigenvalue of A, if $Ax = \lambda x$ for some non zero eigenvector $x = [x_1, x_2, \cdots, x_n]^T$ of quaternions.\\
To estimate the zeros of a polynomial is a long standing classical problem. It is an interesting area of research for engineers as well as mathematicians and many results on the topic are available in the literature. One of the famous result regarding the distribution of zeros of the polynomials  known as  Enestr\"{o}m-Kakeya theorem is as follows:
\begin{thmx}
Let $p(z) = \sum\limits_{j=0}^{n}a_{j}z^{j}$ be a polynomial of degree $n$ such that $0 < a_{0} \leq a_{1} \leq \cdots \leq a_{n}$, then all the zeros of $p(z)$ lie in $\left|z\right| \leq 1$. 
\end{thmx}
In last two years various results were proved by several authors regarding the location of zeros of quaternion polynomials. Recently, Carney et al. \cite{gr} extended Enestr\"{o}m-Kakeya theorem to quaternion settings by proving following result.
\begin{thmx}\label{ek}
If $p(q) = q^na_n+q^{n-1}a_{n-1}+q^{n-2}a_{n-2}+...+qa_1+a_0$ is a polynomial of degree $n$ (where $q$ is a quaternionic
variable) with real coefficients satisfying $0 \leq a_0 \leq a_1 \leq ... \leq a_n,$ then all the zeros of $p$ lie
in $|q| \leq 1.$
\end{thmx}
For complex case, concerning the location of the zeros, the famous Cauchy's theorem \cite{MM} can be stated as:
\begin{thmx}
If $p(z) = \sum\limits_{j=0}^{n}a_{j}z^{j}$ is a polynomial of degree $n$ where $a_{n} \neq 0$ with complex coefficients, then all the zeros of $p(z)$ lie in $\left|z\right| \leq 1 + M$, where $M = \underset{0\leq j\leq{n-1}}{Max}\left|\frac{a_{j}}{a_{n}}\right|$. 
\end{thmx}
Recently, Dar et al. \cite{RD} proved the following quaternion version of Cauchy's theorem.
\begin{thmx}\label{L1}
If $f(q) = q^{n} +q^{n-1}a_{1}+ \cdots +qa_{n-1} + a_n$ is a quaternion polynomial with quaternion coefficients and $q$ is quaternionic variable, then all the zeros of $f(q)$ lie inside the ball \quad $|q|\leq 1+ \underset{1\leq {\nu} \leq{n}}{Max}|a_{\nu}|.$
\end{thmx}
The theorem $D$ was refined by Rather et al. \cite{NR} by proving:
\begin{thmx}\label{od}
Let $f(q) = q^n + q^{n-1}a_1 + q^{n-2}a_2 + \cdots + qa_{n-1} + a_{n}$ be a monic quaternion polynomial of degree n with quaternion coefficients and q be a quaternion variable. If $\alpha_2$ $\geq$ $\alpha_3$ $\geq$ $\cdots$ $\geq$ $\alpha_n$ are ordered positive numbers,
\begin{equation}\label{d1}
 \alpha_\nu = \frac{|a_\nu|}{r^\nu}, \qquad \nu = 2,3, \cdots ,n.
\end{equation}
where r is a positive real number. Then all the zeros of f(q) lie in the union of balls\\
$\left\lbrace q \in H : |q|\leq r(1+\alpha_2)\right\rbrace$ $\hspace{5 mm}$  and $\hspace{5 mm}$ $\left\lbrace q\in H :|q+a_1| \leq r\right\rbrace.$ 
\end{thmx}

\section{\textbf{Main Results}}
We begin with the following result which is a significant refinement of Theorem {D}.
\begin{theorem}\label{tA}
If  $f(q) = q^n + q^{p}a_{p} + q^{p-1}a_{p-1} + \cdots + qa_{1}+a_{0}$, $0 \leq p \leq {n-1}$ is a monic quaternion polynomial of degree $n$ with quaternion coefficients and q be a quaternion variables, and $\left|a_{\nu}\right| \leq M$, $ \quad \nu = 0, 1,\cdots, {n-p}$, then all the zeros of $f(q)$ lie in the ball
 
 \begin{equation}\label{d2}
 \left|q\right| \leq \left\lbrace\left(1 + M\right)^{p+1} - 1\right\rbrace^{\frac{1}{n}}.
 \end{equation}
\end{theorem}
If in Theorem \ref{tA}, we take $p = {n-1}$, we get following result.
  \begin{corollary}\label{c1}
Let $f(q) = q^n + q^{n-1}a_{n-1} + q^{n-2}a_{n-2} + \cdots + qa_{1} + a_{0}$ be a monic quaternion polynomial of degree n with quaternion coefficients and q be a quaternion variable, and $\left|a_{\nu}\right| \leq M, \quad \nu = 0, 1, \cdots, {n-1},$ then all the zeros of $f(q)$ lie in the ball
\begin{equation}\label{d1}
 \left|q\right| \leq\left\lbrace\left(1 + M\right)^{n} - 1\right\rbrace^{\frac{1}{n}}.
\end{equation}
  \end{corollary}
The following corollary is an immediate consequence of Theorem \ref{tA}.
\begin{corollary}\label{c2}
If  $f(q) = q^n + q^{p}a_{p} + q^{p-1}a_{p-1} + \cdots + qa_{1}+a_{0}$, $0 \leq p \leq {n-1}$ is a quaternion polynomial of degree $n$ with quaternion coefficients, then $ \quad \left|q\right| \leq \left(1 + M\right)^{\frac{p+1}{n}}$. 
\end{corollary}
\begin{remark}\label{r1} 
If $p = n-1$, then Corollary \ref{c1} reduces to Theorem \ref{od}.  
 \end{remark}
 
 \begin{corollary}\label{c3}
 If  $f(q) = q^n + q^{p}a_{p} + q^{p-1}a_{p-1} + \cdots + qa_{1}+a_{0}$, $0 \leq p \leq {n-1}$ is a quaternion polynomial of degree $n$ with quaternion coefficients such that $\left|a_{j}\right| \leq 1$, $j = 0, 1, \cdots, p$, then all the zeros of $f(q)$ lie in the ball $\quad \left|q\right| \leq 2^{\frac{p+1}{n}}$.
 \end{corollary}
 From Corollary $\ref{c2}$, we can easily deduce the following:
 \begin{corollary}\label{c4}
If  $f(q) = q^n + q^{p}a_{p} + q^{p-1}a_{p-1} + \cdots + qa_{1}+a_{0}$, $0 \leq p \leq {n-1}$ is a quaternion polynomial of degree $n$ with quaternion coefficients such that $\left|a_{j}\right| \leq 1$, $j = 0, 1, \cdots, p$, then all the zeros of $f(q)$ lie in the ball $\quad \left|q\right| \leq \left(2^{n} - 1\right)^{\frac{1}{n}}$.
\end{corollary}
As an application to Corollary \ref{c4}, we now present the following result regarding location of zeros of a quaternion polynomial.
  \begin{theorem}\label{tB} 
  Let $f(q) = q^n + q^{n-1}a_{n-1} + q^{n-2}a_{n-2} + \cdots + qa_1 + a_0$ be a monic quaternion polynomial with quaternion coefficients and $\left|f(q)\right|$ attains maximum on $\left|q\right| = t$ at the point $q = te^{t\alpha}$ where $t\in R$, then $f(q)$ does not vanish in the ball
 \begin{equation}\label{d3}
 \left|q - te^{i\alpha}\right| < \frac{t}{n\left(2^{n} - 1 \right)^{\frac{1}{n}}}.
  \end{equation}
  \end{theorem}
\section{\textbf{Lemma}}
For the proof of these theorems we need following lemmas.\\
\indent Lemma $\ref{L1}$ is due to Dar et al. \cite{RD}.
\begin{lemma}\label{L1}
All the left eigenvalues of a $n \times n$ matrix $A=(a_{\mu\nu})$ of quaternions lie  in the union of the n Ger\v{s}gorin balls defined by $B_{\mu}=\{q \in \mathbb{H}:|q-a_{\mu\mu}|\leq \rho_{\mu}(A) \},$ where $\rho_{\mu}(A)=\sum_{{\nu=1}\atop{\nu\neq\mu}}^n |a_{\mu\nu}|.$
\end{lemma}
Lemma $\ref{L2}$ is due to Rather et al. \cite{NR}.
\begin{lemma}\label{L2}
Let $P(q)$ be a quaternion polynomial with quaternionic coefficients and $C_{p}$ be the companion matrix of $P(q)$, then for any diagonal matrix $D= diag(d_1, d_2, ..., d_{n-1}, d_n),$ where $d_1, d_2, ..., d_n$ are positive real numbers, the left eigenvalues of $D^{-1}C_P D$ and the zeros of $P(q)$ are same.
\end{lemma}
\begin{lemma}\label{L3}
If $f(q) = q^n + q^{p}a_{p} + \cdots + qa_{1} + a_{0}$, $0 \leq p \leq {n-1}$ is a quaternion polynomial of degree $n$ and if $\delta_{1}, \delta_{2}, \cdots, \delta_{p+1}$ are $p+1$ non-zero quaternion numbers such that $\sum\limits_{k=1}^{p+1} \left|\delta_{k}\right| \leq 1$, then all the zeros of $f(q)$ lie in the ball $\quad \left|q\right| \leq R$, where $R = \left\lbrace\underset{1\leq k\leq {p+1}}{Max} \frac{\left|{a_{p-k+1}}\right|}{\left|\delta_{k}\right|}\right\rbrace^{\frac{1}{n-p+k-1}}$.
\end{lemma}
\begin{proof}[\bf{ Proof of Lemma \ref{L3}}] 
The companion matrix for the polynomial $f(q) = q^n + q^{p}a_{p} + \cdots + qa_{1} + a_{0}$, $0 \leq p \leq {n-1}$ is given by
$$C_f=\begin{bmatrix}
0& 0& 0& \cdots & 0& -a_{0}\\
1& 0& 0& \cdots & 0& -a_{1}\\
0& 1& 0& \cdots & 0& -a_{2}\\

\vdots\\
0& 0& 0& \cdots& 1& -a_{p}\\
\vdots\\
0& 0& 0& \cdots& 1& 0\\ 
\end{bmatrix}.$$ 
We take matrix $P = diag\left(\frac{1}{r^{n-1}}, \frac{1}{r^{n-2}},\cdots, \frac{1}{r}, 1\right)$, where $r$ is a positive real number and form the matrix
$$P^{-1}C_{f} P =\begin{bmatrix}
0& 0& 0& \cdots & 0& -\frac{a_{0}}{r^{n-1}}\\
r& 0& 0& \cdots & 0& -\frac{a_{1}}{r^{n-2}}\\
0& r& 0& \cdots & 0& -\frac{a_{2}}{r^{n-3}}\\

\vdots\\
0& 0& 0& \cdots& 0& -\frac{a_{p}}{r^{n-p-1}}\\
\vdots\\
0& 0& 0& \cdots& r& 0\\ 
\end{bmatrix}.$$ 
Applying Lemma \ref{L1} to the matrix $P^{-1}C_{f} P$, it follows that all the left eigenvalues of $P^{-1}C_{f} P$ lie in the union of balls \quad $\left|q\right| \leq t$ and \quad $\left|q + a_{n-1}\right| \leq \frac{|a_{0}|}{r^{n-1}} + \frac{|a_{1}|}{r^{n-2}} + \cdots + \frac{|a_{n-2}|}{r} + \left|a_{n-1}\right|$.\\
Since
\begin{equation}\nonumber
\begin{split}
|q|=|q+a_{n-1}-a_{n-1}| &
\leq |q+a_{n-1}|+|a_{n-1}|\\
&\leq \frac{a_{0}}{r^{n-1}} + \frac{a_{1}}{r^{n-2}} + \cdots + \frac{a_{n-2}}{r} + \left|a_{n-1}\right|\\
&= \displaystyle\sum_{k=1}^{n}\frac{|a_{n-k}|}{r^{k-1}}.
\end{split}
\end{equation}
That is, all the left eigenvalues of the matrix $T^{-1}C_{f}T$ lie in the ball
\begin{align}
|q|\leq max \Big\{r,\displaystyle\sum_{k=1}^{n}\frac{|a_{n-k}|}{r^{k-1}}\Big\}.
\end{align}
We now choose
\begin{align*}
r=max\Big\{\frac{|a_{n-k}|}{|\delta_{k}|}\Big\}^{1/k}, \quad k=1,2,\cdots,n.
\end{align*}
Then 
\begin{align*}
\frac{|a_{n-k}|}{|\delta_{k}|} \leq r^k \delta_{k}, \quad k=1,2,\cdots,n,
\end{align*}
which gives 
\begin{align*}
\frac{|a_{n-k}|}{r^{k-1}} \leq r|\delta_{k}|,
\end{align*}
so that
\begin{align*}
\displaystyle\sum_{k=1}^{n}\frac{|a_{n-k}|}{r^{k-1}}\leq\displaystyle\sum_{k=1}^{n}r|\delta_k|=t\displaystyle\sum_{k=1}^{n}|\delta_k|\leq r.
\end{align*}
Using this in $(5)$, it follows that all the left eigenvalues of the matrix $T^{-1}C_{f}T$ lie in
\begin{align}
\left|q\right| \leq \max\limits_{1 \leq k \leq n}\left\lbrace\frac{\left|a_{n-k}\right|}{\left|\delta_{k}\right|}\right\rbrace^{\frac{1}{k}}.
\end{align}
Since T is a diagonal matrix with real positive entries, by Lemma $2$, it follows that the left eigenvalues of $T^{-1}C_{f}T$ are the zeros of $f(q)$. Therefore, all the zeros of $f(q)$ lie in the ball given by $(6)$.\\
This completes the proof of Lemma $\ref{L3}$..
\end{proof}
Lemma \ref{L4} is due to Zhenghua \cite{zx}.
\begin{lemma}\label{L4}
If $f(q) = q^n + q^{p}a_{p} + q^{p-1}a_{p-1} + \cdots + qa_{1}+a_{0}$ is a quaternion polynomial of degree $n $ and $1 \leq p \leq \infty$, then
\begin{align*}
\underset{\left|q\right|=r}{Max}\left|f^{\prime}(q)\right| \leq n \underset{\left|q\right|=r}{Max}\left|f(q)\right|.
\end{align*}
\end{lemma}
Applying Lemma \ref{L4} to the quaternion polynomial $f(rq)$, where $r$ is any positive real number, we get:
\begin{lemma}\label{L5}
If $f(q) = q^n + q^{p}a_{p} + q^{p-1}a_{p-1} + \cdots + qa_{1}+a_{0}$ is a quaternion polynomial of degree $n $ and $1 \leq p \leq \infty$, then
\begin{align*}
\underset{\left|q\right|=r}{Max}\left|f^{\prime}(q)\right| \leq \frac{n}{r} \underset{\left|q\right|=r}{Max}\left|f(q)\right|.
\end{align*}
\end{lemma}
The next lemma is obtained by repeated application of Lemma \ref{L5}.
\begin{lemma}\label{L6}
If $f(q)$ is a quaternion polynomial of degree $n \geq 1$, and $r$, is any positive real number, then 
\begin{align*}
\underset{\left|q\right|=r}{Max}\left|f^{\prime}(q)\right| \leq \frac{n(n-1)\cdots(n-k+1)}{r^{k}} \underset{\left|q\right|=r}{Max}\left|f(q)\right|, \quad k = 1, 2, \cdots, n.
\end{align*}
\end{lemma} 
\section{\textbf{Proof of the main theorems}}
\begin{proof}[\bf{ Proof of Theorem \ref{tA}}]
By hypothesis, we have
\begin{equation}
\left|q_{p-k+1}\right| \leq M, \quad  k = 1,2,\cdots,{p+1}.
\end{equation}
We take
\begin{align}\label{t2e1}
\delta_{k} = \left[\frac{\left(1 + M\right)^{n}}{\left(1 + M\right)^{p+1} - 1}\right]\left[\frac{q_{p-k+1}}{\left(1 + M\right)^{n-p+k-1}}\right].
\end{align}
Then with the help of (1), we get
\begin{align}\nonumber
\sum_{k=1}^{p+1}\left|\delta_{k}\right| = \frac{\left(1 + M\right)^{n}}{\left(1 + M\right)^{p+1}-1}\sum_{k=1}^{p+1}\left|q_{p-k+1}\right|\frac{1}{\left(1 + M\right)^{n-p+k-1}}
\end{align}
\begin{align}
\leq \frac{\left(1 + M\right)^{n}}{\left(1 + M\right)^{p+1} - 1}\sum_{k=1}^{p+1}\frac{M}{\left(1 + M\right)^{n-p+k-1}}.
\end{align}
Now
\begin{equation}
\begin{split}
\sum_{k=1}^{p+1}\frac{M}{\left(1 + M\right)^{n-p+k-1}} &= \frac{M}{\left(1 + M\right)^{n-p}}\sum_{k=1}^{p+1}\frac{M}{\left(1 + M \right)^{k-1}}\\
&= \frac{M}{\left(1 + M\right)^{n-p}}\left[\frac{1 - \frac{1}{\left(1+M\right)^{p+1}}}{1 - \frac{1}{\left(1+M\right)}}\right]\\
&=\frac{\left(1 + M\right)^{p+1} - 1}{\left(1 + M\right)^{n}}. 
\end{split}
\end{equation}
Using $(10)$ in $(9)$, we obtain $\sum\limits_{k=1}^{p+1}\left|\delta_{k+1}\right| \leq 1$.
Applying Lemma \ref{L3} with $\delta_{k}$, $k = 1, 2,\cdots, {p+1}$ defined by (2), it follows that all the zeros of $f(q)$ lie in the ball
\begin{equation}\nonumber
\begin{split}
\left|q\right| &\leq \underset{1\leq k \leq {p+1}}{Max}\left|\frac{1}{\delta_{k}}q_{p-k+1}\right|^{\frac{1}{n-p+k-1}}\\
&= \underset{1\leq k\leq{p+1}}{Max}\left[\frac{\left(1 + M\right)^{n-p+k-1}\left\lbrace\left(1 + M\right)^{p+1} - 1\right\rbrace}{\left(1 + M\right)^{n}}\right]^{\frac{1}{n-p+k-1}}\\
&= \left(1 + M\right) \underset{1\leq k\leq{p+1}}{Max}\left[\frac{\left(1 + M\right)^{p+1} - 1}{\left(1 + M\right)^{n}}\right]^{\frac{1}{n-p+k-1}}\\
&= \left(1 + M \right)\left[\frac{\left(1 + M\right)^{p+1} - 1}{\left(1 + M\right)^{n}}\right]^{\frac{1}{n}}\\
&= \left[\left(1 + M\right)^{p+1} - 1\right]^{\frac{1}{n}}.
\end{split}
\end{equation}
This completes the proof of Theorem \ref{tA}.
\end{proof}
\begin{proof}[\bf{ Proof of Theorem \ref{tB}}]
Let $t$ be any positive real number and let $w = te^{i\alpha}$, $\alpha \in R$. Then by hypothesis 
\begin{align*}
\underset{\left|q\right|=t}{Max}f(q) = \left|q\left(te^{i\alpha}\right)\right| = \left|q(w)\right|.
\end{align*}
Now consider a polynomial
\begin{equation}\nonumber
\begin{split}
R(q) &= f\left(\frac{t}{n}q + w\right)\\
&= f(w) + \left(\frac{t}{n}\right)f^{\prime}(w)q + \left(\frac{t}{n}\right)^{2} f^{\prime\prime}(w)\frac{q^{2}}{2{!}} + \cdots + \left(\frac{t}{n}\right)^{n} f^{n}(w)\frac{q^{n}}{n{!}}.
\end{split}
\end{equation}
If $T(q) = q^{n}R\left(\frac{1}{q}\right)$, then we have
\begin{equation}\nonumber
\begin{split}
T(q) &= f(w)q^{n} + \left(\frac{q}{n}\right)f^{\prime}(w)q^{n-1} + \cdots + \left(\frac{t}{n}\right)^{n}\frac{f^{n}(w)}{n!}\\
&= \sum\limits_{j=0}^{n}\left(\frac{t}{n}\right)^{n-j}\frac{q^{(n-j)}(w)q^{j}}{(n-j)!}.
\end{split}
\end{equation}
Since $w = te^{i\alpha}$, by using Lemma \ref{L6}, we obtain
\begin{equation}\nonumber
\begin{split}
\left|f^{(n-j)}(w)\right| &= \left|f^{(n-j)} (te^{i\alpha})\right|\\
&\leq \frac{n(n-1)\cdots(j+1)}{t^{n-j}}\underset{\left|q\right|=t}{Max}\left|f(q)\right|\\
&\leq \frac{n(n-1)\cdots(j+1)}{t^{n-j}}\underset{\left|q\right|=t}{Max}\left|f(q)\right|\\
&\leq \frac{n(n-1)\cdots(j+1)}{t^{n-j}}\left|f(w)\right|\\
&\leq \frac {n^{n-j}}{t^{n-j}}\left|f(w)\right|\\
&= \left(\frac{n}{j}\right)^{(n-j)}\left|f(w)\right|, \quad j = 0, 1, \cdots,(n-1).
\end{split}
\end{equation} 
This implies
\begin{equation}\nonumber
\begin{split}
\left|\left(\frac{t}{n}\right)^{n-j}\frac{f^{(n-j)}(w)}{(n-j)!}\right| &= \left(\frac{t}{n}\right)^{n-j}\frac{\left|f^{(n-j)}(w)\right|}{(n-j)!}\\
&\leq \left(\frac{t}{n}\right)^{(n-j)}\left|f^{(n-j)}(w)\right|\\
&\leq \left|f(w)\right|, \quad j = 0, 1, \cdots, (n-1).
\end{split}
\end{equation}
Which shows that the polynomial $T(q)$ satisfies the conditions of Corollary \ref{c4}. Consequently all the zeros of $T(q)$ lie in the ball $\left|q\right| \leq \left(2^{n} - 1\right)^{\frac{1}{n}}$.
Since $R(q) = q^{n}T\left(\frac{1}{q}\right)$, it follows that all the zeros of $T(q)$ lie in the ball $\quad \left|q\right| \geq \frac{1}{\left(2^{n} - \right)^{\frac{1}{n}}}$.
Replacing $q$ by $\left(q - w\right)\left(\frac{n}{t}\right)$ and noting that $f(q) = R\left(q - w\right)\left(\frac{n}{t}\right)$, we conclude that the quaternion polynomial $f(q)$ does not vanish in the ball $\quad \left|q - w\right| < \frac{t}{n\left(2^{n} - 1\right)^{\frac{1}{n}}}$, which is the desired result.\\
This completes the proof of Theorem \ref{tB}.\\
\end{proof}
\section{Declarations}
\subsection{Availability of data and material}
Not applicable.
\subsection{Competing interests}
The authors declare that they have no competing interests.
\subsection{Funding}
None.
\subsection{Author's contributions}
All authors contributed equally to the writing of this paper. All authors read and approved the final manuscript.

\end{document}